\documentclass[10pt]{amsart}
\usepackage{amsthm,amssymb,amsmath}

\newtheorem{theorem}{Theorem}[section]
\newtheorem{lemma}[theorem]{Lemma}

\newtheorem{corollary}[theorem]{Corollary}

\theoremstyle{definition}

\begin{document}

\title[Dependence of solution on a parameter]{On uniform continuous dependence of solution of Cauchy problem on a parameter}

\author{V.~Ya.~Derr}

\address{Faculty of Mathematics, Udmurtia State University, Russia}

\email{derr@uni.udm.ru}

\begin{abstract}
Suppose that an $n$-dimensional Cauchy problem
$$\dfrac{dx}{dt}=f(t,x,\mu)\quad (t\in \mathcal I,\;\mu\in \mathcal
M),\quad x(t_0)=x^0$$ satisfies the conditions that guarantee existence, uniqueness and continuous dependence of solution $x(t,t_0,\mu)$ on parameter $\mu$ in an open set $\mathcal M$.
We show that if one additionally 
requires that family
$\{f(t,x,\cdot)\}_{(t,x)}$ is equicontinuous, then the dependence of solution
$x(t,t_0,\mu)$ on parameter $\mu \in \mathcal M$ is
uniformly continuous. 

An analogous result for a linear $n\times
n$-dimensional Cauchy problem
$$\dfrac{dX}{dt}=A(t,\mu)X+\Phi(t,\mu)\quad (t\in \mathcal I,\;\mu\in
\mathcal M),\quad X(t_0,\mu)=X^0(\mu)$$
is valid under the assumption that the integrals 
$\int_{\mathcal I} \|A(t,\mu_1)-A(t,\mu_2)\|\,dt $ and $\int_{\mathcal I}
\|\Phi(t,\mu_1)-\Phi(t,\mu_2)\|\,dt$ can be made smaller than any given constant (uniformly with respect to $\mu_1$, $\mu_2 \in \mathcal M$)
provided that $\|\mu_1-\mu_2\|$ is
sufficiently small.
\end{abstract}


\maketitle

\section{Introduction}

Let $\mathcal I\subset \mathbb R$ be an open interval, let $\mathcal X\subset \mathbb R^n, \mathcal M \subset \mathbb R^m$ be domains (i.e.~open connected subsets). We set 
$\mathcal D\doteq \mathcal I\times \mathcal X$, $\mathcal G\doteq \mathcal I\times \mathcal M$, $\mathcal O\doteq \mathcal I\times \mathcal X\times \mathcal M$. Suppose that we are given $f:\mathcal O\to \mathbb R^n,\;t_0\in\mathcal I,\;x^0\in \mathcal X.$  

We consider Cauchy problem
\begin{equation}
\label{ex1}
\dfrac{dx}{dt}=f(t,x,\mu)\quad (t\in \mathcal I,\;\mu\in \mathcal M),\quad x(t_0)=x^0.
\end{equation}
The questions related to existence and uniqueness of solution of (\ref{ex1}), its extension by continuity up to the boundary of $\mathcal D$ (to the maximal interval of existence), and its continuous dependence on a paramter $\mu\in\mathcal M$ are discussed, e.g.~in \cite[p.53--73]{coddl} (see also \cite[p.\,19--28,\,119]{hartm}). The next theorem contains a number of basic results on (\ref{ex1}) that can be found in \cite{coddl}.

\begin{theorem}
\label{thm1}
Suppose that function $f$ satisfies:

1) $f$ is measurable on $\mathcal O;$ 

2) $f$ is continuous in $(x,\mu)\in \mathcal X\times \mathcal M$ for every fixed $t\in \mathcal I;$

3) there exists a Lebesgue locally summable function $m$ on $\mathcal I$ such that
$$
\|f(t,x,\mu)\|\le m(t)\qquad  \bigl((x,\mu)\in \mathcal X\times \mathcal M\bigr);
$$

4) for almost every $t\in\mathcal I$ and every $\mu\in\mathcal M$ function $f$ satisfies the Lipschitz condition in $x$:
$$
\|f(t,x',\mu)-f(t,x'',\mu)\|\le L\|x'-x''\|\;\;x',x''\in\mathcal X,
$$ 
where Lipschitz constant $L$ is independent of $t$ and $\mu.$

Then there exists a closed interval $[a,b]\subset \mathcal I$ such that for every $\mu\in\mathcal M$ problem $(\ref{ex1})$ has the unique solution $x(t,\mu)$ on $[a,b]$ that is absolutely continuous in variable $t$, and depends 
continuously on $\mu\in\mathcal M$.
{\rm (}$\mathcal I$ is the maximal interval of existence of $x$.{\rm )}
\end{theorem}

Here and below $\|\cdot\|$ denotes a norm in $\mathbb R^k$. We use the same notation for
norm of a matrix with entries in $\mathbb R$.

The study of ordinary differential equations in the space of Colombeau generalized functions (see \cite{colombo}) requires that the solution $x$ of (\ref{ex1}) depends on \textit{uniformly continuous} (cf.~Theorem \ref{thm1}, where the dependence of $x$ on $\mu$ is only \textit{continuous}). We note that it is essential for our purposes that $\mathcal M$ is an open subset (for a closed and bounded $\mathcal M$ the uniform continuous dependence of $x$ on $\mu$ would follow trivially from Cantor's theorem).

That under the assumptions of Theorem \ref{thm1} the dependence of $x$ on $\mu$ is not necessarily uniformly continuous is demonstrated by the following example. Consider Cauchy problem $$\dfrac{dx}{dt}=\sin\,\frac{1}{\mu},\;x(0)=0,\;\;t\in [0,1],\;\mu\in \mathcal M\doteq (0,1).$$ The assumptions of Theorem \ref{thm1} are satisfied: indeed, given $t>0$ and $\delta$, denote $$\mu_1=\dfrac{1}{\pi n},\;\mu_2=\dfrac{1}{\pi n+\frac{\pi}{2}}.$$ Then
$$|\mu_1-\mu_2|=\dfrac{1}{\pi n(2n+1)}<\delta$$ for a sufficiently large $n,$ and $$|x(t,\mu_1)-x(t,\mu_2)|=\left|t\sin\,\frac{1}{\mu_1}-t\sin\,\frac{1}{\mu_2}\right|=~t.$$ The solution $x(t,\mu)=t\sin\,\frac{1}{\mu}\;\;(t\in [0,1])$, however, is not uniformly continuous on $\mathcal M=(0,1).$

\medskip

\textit{The following question naturally arises: what additional assumptions are required (cf.~Theorem \ref{thm1}) in order to ensure the \textit{uniform continuous} dependence of solution $x(t,\mu)$ on parameter $\mu$? 
Our answer to this question is proposed below.}

\section{Nonlinear Cauchy problem}
\label{s1}


Suppose that we are given a map $F:\mathcal O\to \mathbb R^{n\times p}.$
We say that the family  $\{F(t,x,\cdot)\}_{(t,x)\in \mathcal D}$ of maps $\mathcal M\to \mathbb R^{n\times p}$ is equicontinuous if
\begin{equation}\label{ex2}
(\forall \varepsilon >0)\;(\exists \delta >0)\;\;\bigl(\forall (t,x)\in\mathcal D,\;\forall \mu_1,\mu_2\in \mathcal M:\;\|\mu_1-\mu_2\|<\delta\bigr)\;\;\;\bigl(\|F(t,x,\mu_1)-F(t,x,\mu_2)\|<\varepsilon\bigr).
\end{equation}
For example, if  $F(t,x,\mu)=g(t,x)h(\mu),$ where $g:\mathcal D\to \mathbb R^{n\times m},\;h:\mathcal M\to \mathbb R^{m\times p},$ function $g$ is continuous and bounded on $\mathcal D,$ then $\{F(t,x,\cdot)\}_{(t,x)\in \mathcal D}$ 
is equicontinuous if and only if function $h$ is uniformly continuous on $\mathcal M.$ 

\begin{lemma}
\label{l1}
Suppose that $F$ satisfies Lipschitz condition in $\mu\in\mathcal M$ uniformly with respect to $(t,x)\in\mathcal D,$ i.e.~
$\|F\bigl(t,x,\mu_1\bigr)-F\bigl(t,x,\mu_2\bigr)\|\le M\|\mu_1-\mu_2\|\;\;((t,x)\in \mathcal D),$ where $M$ is independent of $t$ and $x.$  Then  $\{F(t,x,\cdot)\}_{(t,x)\in \mathcal D}$ is equicontinuous.
\end{lemma}

\begin{proof}
Let $\varepsilon>0$ be arbitrary, let $\mu_1,\,\mu_2\in \mathcal M$ be such that $\|\mu_1-\mu_2\|<\delta \doteq\frac{\varepsilon}{M}.$ Then $\|F\bigl(t,x,\mu_1\bigr)-F\bigl(t,x,\mu_2\bigr)\|\le M\|\mu_1-\mu_2\|<\varepsilon.$
\end{proof}

Let us note that family  $\{F(t,x,\cdot)\}_{(t,x)\in \mathcal D}$ can be equicontinuous even if function $F$ does not satisfy the Lipschitz condition in variable $\mu.$ 
For instance, if in the example above we set $m=p=1, h(\mu)\doteq \mu\sin\frac{\pi}{\mu},\mathcal M=(0,1),$ then $h$ and, consequently, $F$ do not satisfy the Lipschitz condition in variable $\mu$ on $\mathcal M,$ 
although $h$ is uniformly continuous on $(0,1)$, hence the corresponding function family is equicontinuous.

\begin{theorem} 
\label{thm2}
Suppose that function $f$ satisfies conditions $1)$ --- $4)$ of Theorem \ref{thm1} and, additionally, condition $5)$: family $\{f(t,x,\cdot)\}_{(t,x)\in \mathcal D}$ is equicontinuous on $\mathcal M.$

Then the dependence of the solution $x(t,\mu)$ of Cauchy problem $(\ref{ex1})$ on $\mu \in \mathcal M$ is uniform continuous {\rm(}uniformly with respect to variable $t\in [a,b]${\rm)}.
\end{theorem}

\begin{proof}
Let us fix an arbitrary $\varepsilon >0$. Suppose that $\delta >0$ in accordance with condition (\ref{ex2}), and let
 $\mu_1,\mu_2\in\mathcal M$ be such that $\|\mu_1-\mu_2\|<\delta.$

Since $x(t,\mu_i)=x^0+\int\limits_{t_0}^t f\bigl(s,x(s,\mu_i),\mu_i\bigr)\,ds\;\;(i=1,2)$ then, assuming first that $t>t_0$ and using conditions 4) and 5), we obtain an estmiate
\begin{multline*}
\|x(t,\mu_1)-x(t,\mu_2)\|\le \int_{t_0}^t \bigl\|f\bigl(s,x(s,\mu_1),\mu_1\bigr)-f\bigl(s,x(s,\mu_2),\mu_2\bigr)\bigr\|\,ds\le \\   \le
\int_{t_0}^t \bigl\|f\bigl(s,x(s,\mu_1),\mu_1\bigr)-f\bigl(s,x(s,\mu_2),\mu_1\bigr)\bigr\|\,ds+\int_{t_0}^t \bigl\|f\bigl(s,x(s,\mu_1),\mu_2\bigr)-f\bigl(s,x(s,\mu_2),\mu_2\bigr)\bigr\|\,ds < \\ <
L\int_{t_0}^t\|x(t,\mu_1)-x(t,\mu_2)\|\,ds+\varepsilon (b-a).
\end{multline*}
Now, using Gronwall-Bellman inequality \cite[p. 37]{hartm} we obtain 
$$
\|x(t,\mu_1)-x(t,\mu_2)\|<\varepsilon (b-a)e^{L(b-a)}.
$$
In the case $t<t_0$ the argument is analogous. The obtained inequality immediately yields the uniform continuous dependence of $x$ on $\mu \in \mathcal M$ {\rm(}uniformly with respect to $t\in [a,b]${\rm)}.
\end{proof}

Using Lemma  \ref{l1}, we obtain
\begin{corollary}
\label{sl1}
{\it Suppose that function $f$ satisfies conditions $1)$--$4)$ of Theorem \ref{thm1} and a Lipschitz condition
$$
\bigl\|f(t,x,\mu_1)-f(t,x,\mu_2)\bigr\|\le M\|\mu_1-\mu_2\|\quad \bigl((t,x)\in \mathcal D,\;\;\mu_1,\mu_2\in\mathcal M\bigr),
$$
where $M$ is independent of $t$ and $x.$ Then the assertion of Theorem \ref{thm2} holds.}
\end{corollary}

By the remark above the additional assumption 5) of Theorem \ref{thm2} is weaker than the additional assumption of Corollary \ref{sl1}. 
Indeed, for the Cauchy problem
$$\dfrac{dx}{dt}=\mu\sin\,\frac{1}{\mu},\;x(0)=0,\;\;t\in [0,1],\;\mu\in \mathcal M\doteq (0,1)$$ the assumptions of Theorem \ref{thm2} are satisfied, while the assumptions of Corollary \ref{sl1} are not.

\section{Linear Cauchy problem}
\label{s2}

We now consider the linear variant of problem (\ref{ex1}):
\begin{equation}\label{linex1}
\dfrac{dx}{dt}=A(t,\mu)x+\varphi(t,\mu)\quad (t\in \mathcal I,\;\mu\in \mathcal M),\quad x(t_0)=x^0.
\end{equation}
The solution of (\ref{linex1}) exists on the whole interval $\mathcal I$ and is possibly unbounded. Thus, in general one can not expect that family
$$
\{f(t,x,\cdot)\}_{(t,x)\in \mathcal D}=\{A(t,\cdot)x+\varphi(t,\cdot)\}_{(t,x)\in \mathcal D}
$$ 
is equicontinuous. Of course, we can restrict this family to a closed subinterval $[a,b]\subset \mathcal I$; then solution $x$ of (\ref{linex1}) is bounded on $[a,b]$ and we can apply Theorem \ref{thm2} and Corollary \ref{sl1}. It is, however, desirable to have a sufficient condition that ensures the uniform continuous dependence of
$x$ on a parameter, when the argument $t$ varies in the whole interval $\mathcal I.$

\vskip 10pt
It will be convenient to consider a matrix-valued analogue of problem (\ref{linex1}):
\begin{equation}\label{linex16}
\dfrac{dX}{dt}=A(t,\mu)X+\Phi(t,\mu)\quad (t\in \mathcal I,\;\mu\in \mathcal M),\quad X(t_0,\mu)=X^0(\mu),
\end{equation}
where $A,\,\Phi:\mathcal I\times \mathcal M\to \mathbb R^{n\times n}$ are summable in $t$ over $\mathcal I$ for every $\mu\in \mathcal M,$ $X:\mathcal I\times \mathcal M\to \mathbb R^{n\times n}$ is absolutely continuous in $t$ over $\mathcal I$ for every $\mu\in \mathcal M$, and $X^0:\mathcal M\to \mathbb R^{n\times n}$ 
is continuous and bounded on $\mathcal M.$

Let $\mathfrak X$ be the set of absolutely continuous on $\mathcal I$ for all $\mu\in \mathcal M$  $n\times n$-matrices $X(t,\mu)$  endowed with metric
$$
\rho\bigl(X(t,\mu_1),\,X(t,\mu_2)\bigr)\doteq \widehat{\rho}(\mu_1,\mu_2)\doteq \|X(t_0,\mu_1)-X(t_0,\mu_2)\|+
\int_{\mathcal I}\|\dot{X}(t,\mu_1)-\dot{X}(t,\mu_2)\|\,dt.
$$ 
This is a complete metric space. We denote by $\mathfrak X_{t_0}\subset \mathfrak X$ the subspace of non-degenerate $n\times n$-matrices $X(t,\mu)$ normed at point $t_0$ by the condition  $X(t_0,\mu)=E$, where $E$ is the identity matrix.
The induced metric in $\mathfrak X_{t_0}$ is given by the formula
$$
\rho\bigl(X(\cdot,\mu_1),\,X(\cdot,\mu_2)\bigr)\doteq \widehat{\rho}(\mu_1,\mu_2)\doteq
\int_{\mathcal I}\|\dot{X}(t,\mu_1)-\dot{X}(t,\mu_2)\|\,dt.
$$ 
Further, let $\mathfrak A$ be the set of summable on $\mathcal I$ for all $\mu \in \mathcal M$ $n\times n$ matrices $A(t,\mu)$ endowed
with norm
$$\mathfrak n (A)=\;\bigl(\widehat{\mathfrak n}(\mu)\bigr)\;=\displaystyle{\int_{\mathcal I}}\|A(t,\mu)\|\,dt,$$ 
let $\mathfrak A_0$ be the space of bounded and continuous  on $\mathcal M$ $n \times n$ matrices with norm
$\|X^0\|\doteq \underset{\mu\in \mathcal M}{\sup}\|X^0(\mu)\|$.
Clearly, spaces $\mathfrak A$, $\mathfrak A_0$ are Banach.
It follows from Theorems \ref{thm1} and \ref{thm2} that Cauchy problem (\ref{linex16})
determines a continuous map 
$\mathfrak F(\mu):\mathfrak A\times \mathfrak A\times \mathfrak A_0\to\mathfrak X$ 
that depends continuously on $\mu \in \mathcal M$. Similarly, Cauchy problem
\begin{equation}\label{linex16o}
\dfrac{dX}{dt}=A(t,\mu)X\qquad (t\in \mathcal I,\;\mu\in \mathcal M),\quad X(t_0,\mu)=X^0(\mu)
\end{equation}
can be viewed as a continuous (bijective) map
$\mathfrak F_0(\mu):\mathfrak A\times \mathfrak A_0\to\mathfrak X_{t_0}$  
that depends on $\mu \in \mathcal M$ continuously.
Our goal is to establish the conditions under which these maps depend on $\mu$ \textit{uniformly} continuously on $\mathcal M$.

\medskip

We say that $B \in \mathcal A$ is \textit{integrally uniformly continuous} on $\mathcal M$ if the following condition is satisfied:
$$
(\forall \varepsilon >0)\;\;(\exists \delta>0)\;\;(\forall \mu_1,\mu_2\in\mathcal M:
\|\mu_1-\mu_2\|<\delta)\;\;\;\left(\int_{\mathcal I}\|B(t,\mu_1)-B(t,\mu_2)\|\,dt<\varepsilon\right).
$$
\begin{theorem} 
\label{thm3}
Suppose that 

1) functions $\widehat{\mathfrak n},\,\varphi,\,\eta: \mathcal M\to [0,+\infty),$ where $\varphi (\mu)=\displaystyle{\int_{\mathcal I}}\|\Phi(t,\mu)\|\,dt,\quad \eta (\mu)\doteq \|X^0(\mu)\|,$  
are bounded on $\mathcal M$;

2) function $X^0\in \mathfrak A_0$ is uniformly continuous, and functions $A,\,\Phi \in \mathfrak A$ are integrally uniformly continuous on $\mathcal M.$ 

Then the solution $Y(t,\mu)$ of problem $(\ref{linex16})$ is uniformly continuous
in $\mu \in \mathcal M$ (uniformly with respect to $t\in \mathcal I$), and maps $\mathfrak F_0(\mu)(A,\,X^0)$ and $\mathfrak F(\mu)(A,\,\Phi,\,X^0)$ are uniformly continuous on $\mathcal M$
{\rm(}uniformly with respect to $t\in \mathcal I${\rm)}.
\end{theorem}

\begin{proof}
We choose a constant $K>0$ such that the following inequalities are satisfied:
\begin{equation}\label{linex16ner}
\widehat{\mathfrak n}(\mu)\leq K,\quad \varphi (\mu)\leq K,\quad \eta (\mu)\leq K,\quad \xi\doteq \|E\|\leq K.
\end{equation}
Denote
\begin{multline*}
\mathfrak a (\mu_1,\mu_2)\doteq\int_{\mathcal I} \|A(t,\mu_1)-A(t,\mu_2)\|\,dt,\quad \mathfrak f (\mu_1,\mu_2)\doteq\int_{\mathcal I} \|\Phi(t,\mu_1)-\Phi(t,\mu_2)\|\,dt,\\
 \mathfrak x (\mu_1,\mu_2)\doteq \|X^0(\mu_1)-X^0(\mu_2)\|.
\end{multline*}

In the next four lemmas we obtain a number of estimates needed to complete the proof of the theorem.
Some of these lemmas (e.g.~Lemmas \ref{lrisch3} and \ref{lrisch5}) are interesting in their own right.

Let $C(t,s,\mu)=X(t,\mu)X^{-1}(s,\mu)$ be the Cauchy matrix of the homogeneous system of differential equations (\ref{linex16o}) (let $X(t,\mu)$
be its fundamental matrix normed at point $t_0$).

\begin{lemma}
\label{lrisch2}
The following estimates hold:
\begin{equation}\label{linex9}
\|X(t,\mu)\|\leq \xi e^{\widehat{\mathfrak n}(\mu)}\leq Ke^K\quad (t\in\mathcal I,\,\mu\in \mathcal M);
\end{equation}
\begin{equation}\label{linex10}
\|X^{-1}(t,\mu)\|\leq \xi e^{\widehat{\mathfrak n}(\mu)}\leq Ke^K\quad (t\in\mathcal I, \,\mu\in \mathcal M);
\end{equation}
\begin{equation}\label{linex11}
\|X(t,\mu_1)-X(t,\mu_2\|\leq \xi^3e^{2\widehat{\mathfrak n}(\mu_1)+\widehat{\mathfrak n}(\mu_2)}\int_{\mathcal I} \|A(s,\mu_1)-A(s,\mu_2)\|\,ds\leq K^3e^{3K}\mathfrak a (\mu_1,\mu_2)
\end{equation}
$(t\in\mathcal I,\,\mu_1,\mu_2\in\mathcal M);$
\begin{equation}\label{linex12}
\|X^{-1}(t,\mu_1)-X^{-1}(t,\mu_2\|\leq \xi^3e^{2\widehat{\mathfrak n}(\mu_1)+\widehat{\mathfrak n}(\mu_2)}\int_{\mathcal I} \|A(s,\mu_1)-A(s,\mu_2)\|\,ds\leq K^3e^{3K}\mathfrak a (\mu_1,\mu_2)
\end{equation}
$(t\in\mathcal I,\,\mu_1,\mu_2\in\mathcal M);$
\begin{equation}\label{linex13}
\|C(t,s,\mu)\|\leq \xi^2e^{2\widehat{\mathfrak n}(\mu)}\leq K^2e^{2K}\quad (t,s\in\mathcal I,\,\mu\in\mathcal M),
\end{equation}
\begin{multline}
\label{linex14}
\|C(t,s,\mu_1)-C(t,s,\mu_2)\|\leq \\  
\leq\xi^4e^{2\widehat{\mathfrak n}(\mu_1)+\widehat{\mathfrak n}(\mu_2)}\bigl(e^{\widehat{\mathfrak n}(\mu_1)}+e^{\widehat{\mathfrak n}(\mu_2)}\bigr)\int_{\mathcal I} \|A(s,\mu_1)-A(s,\mu_2)\|\,ds\leq 2K^4e^{4K}\mathfrak a (\mu_1,\mu_2),
\end{multline}
where $t,s\in\mathcal I,\,\mu_1,\mu_2\in\mathcal M$.
\end{lemma}

\begin{proof}[Proof of Lemma \ref{lrisch2}]
From differential equations
$\dot{X}(t,\mu_i)=A(t,\mu_i)X(t,\mu_i)$ $(i=1,2)$ we obtain the following differential equation (Cauchy problem)
\begin{multline}\label{linex7}
\dfrac{d\bigl(X(t,\mu_1)-X(t,\mu_2)}{dt}=A(t,\mu_1)\bigl(X(t,\mu_1)-X(t,\mu_2)\bigr)+\bigl(A(t,\mu_1)-A(t,\mu_2)\bigr)X(t,\mu_2),\\ \bigl(X(t_0,\mu_1)-X(t_0,\mu_2)\bigr)=0
\end{multline}
which we may consider as a non-homogeneous matrix differential equation with respect to the unknown function
$Y\doteq X(t,\mu_1)-X(t,\mu_2)$, having $A(t,\mu_1)$ as its matrix, with the right-hand side 
$
\bigl(A(t,\mu_1)-A(t,\mu_2)\bigr)X(t,\mu_2),
$ 
and having zero initial value.
Using Cauchy formula, we obtain
$$
Y(t)=\int_{t_0}^t X(t,\mu_1)X^{-1}(s,\mu_1)\bigl(A(s,\mu_1)-A(s,\mu_2)\bigr)X(s,\mu_2)\,ds,
$$
which implies that
\begin{equation}\label{linex8}
\|Y\|\leq \|X(t,\mu_1)\|\int_{t_0}^t\|X^{-1}(s,\mu_1)\|\cdot \|A(s,\mu_1)-A(s,\mu_2)\|\cdot \|X(s,\mu_2)\|\,ds.
\end{equation}
Now, since $X(t,\mu_1)=E+\int\limits_{t_0}^t A(s,\mu_1)X(s,\mu_1)\,ds,$ we have
$\|X(t,\mu_1)\|\leq \xi+\displaystyle{\int_{t_0}^t} \|A(s,\mu_1)\|\cdot \|X(s,\mu_1)\|\,ds.$ By Gronwall-Bellman inequality
(see, e.g.~\cite[p. 37]{hartm})
\begin{equation*}
\|X(t,\mu_1)\|\leq \xi\exp\left(\int_{t_0}^t \|A(s,\mu_1)\|\,ds\right)\leq \xi\exp\left(\int_{\mathcal I} \|A(s,\mu_1)\|\,ds\right) \leq \xi e^{\widehat{\mathfrak n}(\mu_1)}\quad (t\in\mathcal I).
\end{equation*}
Since $$\dfrac{dX^{-1}(t,\mu_1)}{dt}=-X^{-1}(t,\mu_1)A(t,\mu_1),$$ 
we can apply the above argument to function $X^{-1}(t,\mu_1)$, thus arriving to estimate (\ref{linex10}).

Using (\ref{linex9}),\,(\ref{linex10}),\,(\ref{linex8}), we obtain estimate (\ref{linex11}). 
A similar argument gives us (\ref{linex12}).

Now, it follows from (\ref{linex9})--(\ref{linex10}) that
$$
\|C(t,s,\mu)\|\leq \|X(t,\mu)\|\cdot \|X^{-1}(s,\mu)\|\leq \xi^2e^{2\widehat{\mathfrak n}(\mu)}\quad (t,s\in\mathcal I).
$$
Furthermore, using estimates
(\ref{linex11})  and (\ref{linex12}) we get
\begin{multline*}
\|C(t,s,\mu_1)-C(t,s,\mu_2)\|\!=\!\|X(t,\mu_1)\bigl(X^{-1}(s,\mu_1)-X^{-1}(s,\mu_2)\bigr)-\bigl(X(t,\mu_1)-X(t,\mu_2)\bigr)X^{-1}(s,\mu_2)\|\!\leq \\
\leq \|X(t,\mu_1)\|\cdot \|X^{-1}(s,\mu_1)-X^{-1}(s,\mu_2)\|+\|X(t,\mu_1)-X(t,\mu_2)\|\cdot \|X^{-1}(s,\mu_2)\|\leq \\   \leq
\xi^4e^{2\widehat{\mathfrak n} (\mu_1)+\widehat{\mathfrak n} (\mu_2)}\bigl(e^{\widehat{\mathfrak n}(\mu_1)}+e^{\widehat{\mathfrak n}(\mu_2)}\bigr)\int_{\mathcal I} \|A(s,\mu_1)-A(s,\mu_2)\|\,ds\quad (t,s\in\mathcal I).
\end{multline*}
The proof of Lemma \ref{lrisch2} is complete.
\end{proof}

In the next lemma we esimate from above the distance $\rho\bigl(X(\cdot,\mu_1),\,X(\cdot,\mu_2)\bigr)=\widehat{\rho}(\mu_1,\mu_2)$. 
\begin{lemma}
\label{lrisch3}
$$
\widehat{\rho}(\mu_1,\mu_2)\leq \xi e^{\widehat{\mathfrak n}(\mu_1)}\Bigl(\widehat{\mathfrak n}
(\mu_1)\xi^2e^{\widehat{\mathfrak n}(\mu_1)+\widehat{\mathfrak n}(\mu_2)}+1\Bigr)
\int_{\mathcal I} \|A(s,\mu_1)-A(s,\mu_2)\|\,ds\leq 
\bigl(K^4e^{3K}+Ke^K\bigr)\mathfrak a (\mu_1,\mu_2)
$$
$(\mu_1,\mu_2\in\mathcal M).$
\end{lemma}

\begin{proof}[Proof of Lemma \ref{lrisch3}]
Using estimates (\ref{linex7}), (\ref{linex8}) and (\ref{linex11}) we obtain
\begin{multline*}
\widehat{\rho}(\mu_1,\mu_2)=\int_{\mathcal I}\|\dot{X}(t,\mu_1)-\dot{X}(t,\mu_2)\|\,dt=\\=\int_{\mathcal I}\|A(t,\mu_1)\bigl(X(t,\mu_1)-X(t,\mu_2)\bigr)+\bigl(A(t,\mu_1)-A(t,\mu_2)\bigr)X(t,\mu_2)\|\,dt\leq \\
\int_{\mathcal I}\|A(t,\mu_1)\|\,\|(X(t,\mu_1)-X(t,\mu_2)\|\,dt+\int_{\mathcal I}\|A(t,\mu_1)-A(t,\mu_2)\|\,\|X(t,\mu_2)\|\,dt\leq \\
\leq \xi e^{\widehat{\mathfrak n}(\mu_1)}\Bigl(\widehat{\mathfrak n}
(\mu_1)\xi^2e^{\widehat{\mathfrak n}(\mu_1)+\widehat{\mathfrak n}(\mu_2)}+1\Bigr)
\int_{\mathcal I} \|A(s,\mu_1)-A(s,\mu_2)\|\,ds  
\end{multline*}
This completes the proof of Lemma \ref{lrisch3}.
\end{proof}

Next, we obtain estimates for the solution  $Y=Y(t,\mu)$ of non-homogeneous Cauchy problem (\ref{linex16}). 
\begin{lemma}
\label{lrisch4}
We have the following estimates:
\begin{equation}\label{linex18}
\|Y(t,\mu)\|\leq \|X^0(\mu)\|\,\xi e^{\mathfrak n (\mu)}+\xi^2e^{2\mathfrak n (\mu)}\cdot \int_{\mathcal I}\|\Phi(s,\mu)\|\,ds\leq K^2e^K+K^3e^{2K}\quad (t\in\mathcal I,\,\mu\in\mathcal M).
\end{equation}
\begin{multline}
\label{linex19} 
\|Y(t,\mu_1)-Y(t,\mu_2)\|\leq 
\xi e^{\mathfrak n (\mu_2)}\,\|X^0(\mu_1)-X^0(\mu_2)\|+\\+
\xi^3e^{2\mathfrak n (\mu_1)+\mathfrak n (\mu_2)}\left(\|X^0(\mu_1)\|+\xi e^{\mathfrak n (\mu_1)}+e^{\mathfrak n (\mu_2)}
\int_{\mathcal I}\|\Phi(s,\mu_1)\|\,ds\right)\int_{\mathcal I} \|A(s,\mu_1)-A(s,\mu_2)\|\,ds+\\+
 \xi^2e^{2\mathfrak n (\mu_2)}\cdot \int_{\mathcal I}\|\Phi(s,\mu_1)-\Phi(s,\mu_2)\|\,ds \leq \\ 
\leq K_1\mathfrak x (\mu_1,\mu_2)+K_2\mathfrak a (\mu_1,\mu_2)+K_3\mathfrak f (\mu_1,\mu_2)
 \quad (t\in\mathcal I,\,\mu_1,\mu_2\in\mathcal M),
\end{multline}
where $K_1=Ke^K,\;\;K_2=K^4e^{3K}(1+K+e^K),\;\;K_3=K^2e^{2K}$
\end{lemma}

\begin{proof}[Proof of Lemma \ref{lrisch4}]
Using Cauchy formula we obtain
\begin{equation}\label{linex17}
Y(t,\mu)=X(t,\mu)X^0(\mu)+\int_{t_0}^tC(t,s,\mu)\Phi(s,\mu)\,ds \quad (t\in\mathcal I,\,\mu\in\mathcal M).
\end{equation}
Further, according to (\ref{linex9}) and (\ref{linex13}) we have the estimate
\begin{multline*}
\|Y(t,\mu)\|\leq \|X(t,\mu)\|\,\|X^0(\mu)\|+\int_{t_0}^t\|C(t,s,\mu)\|\cdot \|\Phi(s,\mu)\|\,ds\leq \\   
\leq \|X^0(\mu)\|\,\xi e^{\mathfrak n (\mu)}+\xi^2e^{2\mathfrak n (\mu)}\cdot \int_{\mathcal I}\|\Phi(s,\mu)\|\,ds\quad (t\in\mathcal I,\,\mu\in\mathcal M).
\end{multline*}
Now, in virtue of (\ref{linex9}),\,(\ref{linex13}),\,(\ref{linex11}) and (\ref{linex14}) we have
\begin{multline*}
\|Y(t,\mu_1)-Y(t,\mu_2)\|\leq \|X(t,\mu_1)-X(t,\mu_2)\|\,\|X^0(\mu_1)\|+\|X(t,\mu_2)\|\,
\|X^0(\mu_1)-X^0(\mu_2)\|+\\+\underset{(t,s)\in=\mathcal I^2}{\max}\|C(t,s,\mu_1)-
C(t,s,\mu_2)\|\cdot \int_{\mathcal I}\|\Phi(s,\mu_1)\|\,ds+\\+\underset{(t,s)\in=\mathcal I^2}{\max}\|C(t,s,\mu_2)\|\cdot \int_{\mathcal I}\|\Phi(s,\mu_1)-\Phi(s,\mu_2)\|\,ds\leq \\   \leq\|X^0(\mu_1)\|\,\|E\|^3e^{2\mathfrak n (\mu_1)+\mathfrak n (\mu_2)}\int_{\mathcal I} \|A(s,\mu_1)-A(s,\mu_2)\|\,ds+
\|X^0(\mu_1)-X^0(\mu_2)\|\cdot \xi e^{\mathfrak n (\mu_2)}+\\+\xi^4e^{2\mathfrak n (\mu_1)+\mathfrak n (\mu_2)}\bigl(e^{\mathfrak n (\mu_1)}+e^{\mathfrak n (\mu_2)}\bigr)\int_{\mathcal I} \|A(s,\mu_1)-A(s,\mu_2)\|\,ds\cdot \int_{\mathcal I}\|\Phi(s,\mu_1)\|\,ds
+\\+ \xi^2e^{2\mathfrak n (\mu_2)}\cdot \int_{\mathcal I}\|\Phi(s,\mu_1)-\Phi(s,\mu_2)\|\,ds= \xi e^{\mathfrak n (\mu_2)}\,\|X^0(\mu_1)-X^0(\mu_2)\|+\\+
\xi^3e^{2\mathfrak n (\mu_1)+\mathfrak n (\mu_2)}\left(\|X^0(\mu_1)\|+\xi e^{\mathfrak n (\mu_1)}+e^{\mathfrak n (\mu_2)}\int_{\mathcal I}\|\Phi(s,\mu_1)\|\,ds\right)\int_{\mathcal I} \|A(s,\mu_1)-A(s,\mu_2)\|\,ds+\\+
 \xi^2e^{2\mathfrak n (\mu_2)}\cdot \int_{\mathcal I}\|\Phi(s,\mu_1)-\Phi(s,\mu_2)\|\,ds \quad (t\in\mathcal I,\,\mu_1,\mu_2\in\mathcal M),
\end{multline*}
which concludes the proof of Lemma \ref{lrisch4}.
\end{proof}

Now, we estimate from above the distance
$\rho\bigl(Y(\cdot,\mu_1),\,Y(\cdot,\mu_2)\bigr)=\widehat{\rho}(\mu_1,\mu_2).$
\begin{lemma}
\label{lrisch5}
\begin{multline}
\label{linex20}
\widehat{\rho}(\mu_1,\mu_2)\leq \\
\leq \widetilde K_1\,\|X^0(\mu_1)-X^0(\mu_2)\|+\widetilde K_2\,\int_{\mathcal I} \|A(t,\mu_1)-A(t,\mu_2)\|\,dt+\widetilde K_3\,\int_{\mathcal I} \|\Phi(t,\mu_1)-\Phi(t,\mu_2)\|\,dt,
\end{multline}
where $\widetilde K_1=1+K^2e^K,\;\;\widetilde K_2=K^2e^K(1+Ke^K+K^3e^{2K}+2K^3e^{3K}),\;\;\widetilde K_3=1+K^3e^{2K}.$
\end{lemma}

\begin{proof}[Proof of Lemma \ref{lrisch5}]
According to (\ref{linex16}) we have
\begin{multline*}
\widehat{\rho}(\mu_1,\mu_2)=\|X^0(\mu_1)-X^0(\mu_2)\|+\int_{\mathcal I}\|\dot{Y}(t,\mu_1)-\dot{Y}(t,\mu_2)\|\,dt\leq 
\mathfrak x (\mu_1,\mu_2)+\\+\int_{\mathcal I}\|A(t,\mu_1)\|\,\|Y(t,\mu_1)-Y(t,\mu_2)\|\,dt+
\int_{\mathcal I} \|A(t,\mu_1)-A(t,\mu_2)\|\,\|Y(t,\mu_2)\|\,dt+\mathfrak f (\mu_1,\mu_2)\leq \\  
\leq \mathfrak x\, (\mu_1,\mu_2)+K\Bigl(Ke^K\mathfrak x\,(\mu_1,\mu_2)+K^4e^{3K}(1+2e^K)\mathfrak a (\mu_1,\mu_2)+K^2e^{2K}\mathfrak f \,(\mu_1,\mu_2)\Bigr)+\\+ K^2\bigl(e^K+Ke^{2K}\bigr)\mathfrak a \,(\mu_1,\mu_2)
+\mathfrak f (\mu_1,\mu_2)=\widetilde K_1\mathfrak x\, (\mu_1,\mu_2)+\widetilde K_2\mathfrak a (\mu_1,\mu_2)+\widetilde K_3\mathfrak f (\mu_1,\mu_2),
\end{multline*}
as needed.
\end{proof}

Now, we are ready to complete the proof of Theorem \ref{thm3}.
Let $\varepsilon >0$ be arbitrary. By our assumptions there is $\delta>0$ such that if
$\|\mu_1-\mu_2\|<\delta$,  then 
\begin{multline*}
\mathfrak x (\mu_1,\mu_2)< \dfrac{\varepsilon}{3K_1},\quad \mathfrak a (\mu_1,\mu_2)<\dfrac{\varepsilon}{3K_2},\quad \mathfrak f (\mu_1,\mu_2)<\dfrac{\varepsilon}{3K_3},\\ 
\mathfrak x (\mu_1,\mu_2)<\dfrac{\varepsilon}{3\widetilde K_1},\quad 
\mathfrak a (\mu_1,\mu_2)<\dfrac{\varepsilon}{3\widetilde K_2},\quad \mathfrak f (\mu_1,\mu_2)<\dfrac{\varepsilon}{3\widetilde K_3}.
\end{multline*}
This esimate, combined with
Lemma \ref{lrisch4} (estimates (\ref{linex19})) and \ref{lrisch5}), implies that
$$
\|Y(t,\mu_1)-Y(t,\mu)\|< \varepsilon \quad (t\in \mathcal I),\qquad \rho\bigl(Y(\cdot,\mu_1),\,Y(\cdot,\mu_2)\bigr)< \varepsilon.
$$
The proof of Theorem \ref{thm3} is complete.
\end{proof}

\begin{theorem} 
\label{thm4}
The assertion of Theorem \ref{thm3} remains true if we replace in its formulation $\widehat{\mathfrak n},\,\varphi,\,\mathfrak a,\,\mathfrak f$ with, respectively,
\begin{multline*}
\widehat{\mathfrak n}_q (\mu)\doteq \left(\int_{\mathcal I}\|A(\mu)\|^q\,dt\right)^{\frac{1}{q}},\quad \varphi_q(\mu)\doteq \left(\int_{\mathcal I}\|\Phi(\mu)\|^q\,dt\right)^{\frac{1}{q}},\\
\mathfrak a_p(\mu_1,\mu_2)\doteq \left(\int_{\mathcal I}\|A(t,\mu_1)-A(t,\mu_2)\|^p\,dt\right)^{\frac{1}{p}},\quad \mathfrak f_p(\mu_1,\mu_2)\doteq \left(\int_{\mathcal I}\|\Phi (t,\mu_1)-\Phi (t,\mu_2)\|^p\,dt\right)^{\frac{1}{p}}
\end{multline*}
$\left(1<p<+\infty,\;\frac{1}{p}+\frac{1}{q}=1\right),$ 
\end{theorem}

\begin{proof}
In the proof of Theorem \ref{thm3}, in the estimates of the integrals of products one has to apply H\"{o}lder inequality.
\end{proof}

\end{document}